\documentclass[11pt]{article}
\usepackage{geometry}                % See geometry.pdf to learn the layout options. There are lots.
\usepackage{graphicx}
\usepackage{amsmath}
\usepackage{enumerate}
\usepackage{setspace}
\usepackage{amssymb}
\usepackage{amsmath, amsthm}
\usepackage{amsmath}
\usepackage{amsthm}
\usepackage{blkarray}
\usepackage{epstopdf}
\usepackage{footnote}
\newtheorem{thm}{Theorem}[section]
\newtheorem{Proposition}{Proposition}[section]

\newtheorem{Lemma}[thm]{Lemma}

\begin{document}\large{
%%%%%%%%%%%%%%%%%%                            Title                                  %%%%%%%%%%%%%%%%%%%%%%%%%%%%%
\title{Asymptotic Properties of Discrete Minimal 
 $s,\log^t$-Energy Constants and Configurations}
\author{Nichakan Loesatapornpipit\footnote{Results of this article constitute part of Nichakan Loesatapornpipit's senior project under the mentorship of Nattapong Bosuwan at
Mahidol University.} and Nattapong Bosuwan}
\maketitle

\centerline {Department of Mathematics, Faculty of Science, Mahidol University}
\centerline {Rama VI Road, Ratchathewi District, Bangkok 10400, Thailand}
\centerline {Correspondence e-mail : {\tt nattapong.bos@mahidol.ac.th}}
\centerline {Centre of Excellence in Mathematics, CHE}
\centerline {Si Ayutthaya Road, Bangkok 10400, Thailand}

%%%%%%%%%%%%%%%%%%%%%%%%       Abstract                            %%%%%%%%%%%%%%%%%%%%%%%%%%%%%%

\section*{Abstract}
Combining the ideas of  Riesz $s$-energy and $\log$-energy, we introduce the so-called $s,\log^t$-energy. In this paper, we investigate the asymptotic behaviors for $N,t$ fixed and $s$ varying of minimal $N$-point $s,\log^t$-energy constants and configurations of an infinite compact metric space of diameter less than $1$. In particular, we study certain continuity and differentiability properties of minimal $N$-point $s,\log^t$-energy constants in the variable $s$ and we show that in the limits as $s\rightarrow \infty$ and as $s\rightarrow s_0>0,$ minimal $N$-point $s,\log^t$-energy configurations tend to an $N$-point best-packing configuration and a minimal $N$-point $s_0,\log^t$-energy configuration, respectively. Furthermore, the optimality of $N$ distinct equally spaced points on circles in $\mathbb{R}^2$ for some certain $s,\log^t$ energy problems was proved.
%\doublespacing

\quad

 \section*{Keywords}{discrete minimal energy; best-packing; Riesz energy; logarithmic energy.}

%%%%%%%%%%%%%%%%%%%%%%%%%%%%%%%%%%%%%%%%%%
\section{Introduction}

The general setting of discrete minimal energy problem is the following.
Let $(A,d)$ be an infinite compact metric space and $K: A\times A \rightarrow \mathbb{R}\cup \{\infty\}$ be a lower semicontinuous kernel.
For a fixed set of $N$ points $\omega_N\subset A,$ we define the \emph{$K$-energy of $\omega_N$} as follows
$$E_K(\omega_N):=\sum_{\substack{ x\not=y  \\ x,y \in\omega_N}}K(x,y).$$
The \emph{minimal $N$-point $K$-energy of the set $A$} is defined by
$$\mathcal{E}_{K}(A,N):=\min_{\substack{\omega_N \subset A  \\ \#\omega_N=N}} E_K(\omega_N),$$
where $\#\omega_N$ stands for the cardinality of the set $\omega_N.$
A \emph{minimal $N$-point $K$-energy configuration} is a configuration $\omega_N^{K}$ of $N$ points in $A$ that minimizes such energy, namely
\begin{equation*}
E_K(\omega_N^{K})=\min_{\substack{\omega_N \subset A  \\ \#\omega_N=N}} E_K(\omega_N).
\end{equation*}
It is known that $\omega_N^{K}$ always exists and in general $\omega_N^{K}$ may not be unique.

Two important kernels in the theory on  minimal energy are Riesz and logarithmic kernels.
The \emph{(Riesz) $s$-kernel} and \emph{$\log$-kernel} are defined by
\begin{equation}\label{eq2}
K^s(x,y):=\frac{1}{d(x,y)^s},\quad s\geq 0.
\end{equation}
and
$$
K_{\log}(x,y):=\log \frac{1}{d(x,y)},
$$
for all $(x,y)\in A\times A,$
respectively. It is not difficult to check that both kernels are lower semicontinuous on $A\times A.$
The \emph{$s$-energy of $\omega_N$} and the \emph{minimal $N$-point $s$-energy of the set $A$} are
$$
E^s(\omega_N):=E_{K^s}(\omega_N) \quad \textup{and} \quad \mathcal{E}^s(A,N):=\min_{\substack{\omega_N \subset A  \\ \#\omega_N=N}} E^s(\omega_N)
$$
and
we denote by $\omega_N^s:=\omega_N^{K^s}$ and call this configuration a \emph{minimal $N$-point $s$-energy configuration}. Similarly, the \emph{$\log$-energy of $\omega_N$} and the \emph{$N$-point $\log$-energy of the set $A$} are
$$
E_{\log}(\omega_N):=E_{K_{\log}}(\omega_N) \quad \textup{and} \quad \mathcal{E}_{\log}(A,N):=\min_{\substack{\omega_N \subset A  \\ \#\omega_N=N}} E_{\log}(\omega_N)
$$
and
we denote by $\omega_N^{\log}:=\omega_N^{K_{\log}}$ and call this configuration a \emph{minimal $N$-point $\log$-energy configuration}. 

Let us provide a short survey of these two energy problems. 

The study of $s$-energy constants and configurations has a long history in physics, chemistry, and mathematics. Finding the arrangements of $\omega_N^s$ where the set $A$ is the unit sphere $\mathbb{S}^2$ in the Euclidean space $\mathbb{R}^3$ has been an active area since the beginning of the 19th century. The problem is known as the generalized Thomson problem (see \cite{Thomson}  and \cite[Chapter 2.4]{SaffHardinBoroda}). Candidates for $\omega_N^s$ for several numbers of $N$ are available  (see, e.g., \cite{Candidates}).  However, the solutions (with rigorous proofs) are obtainable for handful values of $N$ (see, e.g., \cite{Foppl,Yudin,Andreev,5points}).
For a general compact set $A$ in the Euclidean space $\mathbb{R}^m,$ the study of the distribution of minimal $N$-point $s$-energy configurations of $A$ as $N\rightarrow \infty$ can be founded in \cite{Landkof} and \cite{saff}. In \cite{saff}, it was shown that when $s$ is any fixed number greater than the Hausdorff dimension of $A,$ minimal $N$-point $s$-energy configurations of $A$ are ``good points" to represent the set $A$ in the sense that they are asymptotically uniformly distributed over the set $A$ (see the precise statement in \cite[Theorems 2.1 and 2.2]{saff}).

The $\log$-energy problem has been heavily studied when $A$ is a subset of the Euclidean space $\mathbb{R}^2$ (or $\mathbb{C}$) because it has had a profound influence in approximation theory (see, e.g., \cite{Mhaskar,GoncharRakhmanov1989,Lubinsky,Totik,SaffTotik}). For $A\subset \mathbb{C},$ the points in $\omega_N^{\log}$ are commonly known as Fekete points or Chebyshev points which can be used as interpolation points (see \cite{Trefethen}). The $\log$-energy problem received another special attention when Steven Smale posed Problem \#7 in his book chapter under the title ``Mathematical problems for the next century" \cite{Smale}. The problem \#7 asks for a construction of an algorithm which on input $N \geq 2$
outputs a configuration $\omega_N = \{x_1, \ldots , x_N\}$ of distinct points on $\mathbb{S}^2$ embedded in $\mathbb{R}^3$ such that
$$E_{\log}(\omega_N) - \mathcal{E}_{\log}(\mathbb{S}^2,N) \leq c \log N$$
(where $c$ is a constant independent of $N$ and $\omega_N$) and requires that its running time grows at most polynomially in $N.$ This arose form complexity theory in his joint work with Shub in \cite{Shub}.  In order to answer this question, it is natural to understand the asymptotic expansion of $\mathcal{E}_{\log}(\mathbb{S}^2,N)$ in the variable $N$ (see \cite{Brauchart} for conjectures and the progress).  The problem concerning the arrangements of $\omega_N^{\log}$ on the unit sphere $\mathbb{S}^2$ in $\mathbb{R}^3$ is posed by Whyte \cite{Whyte} in 1952. The Whyte's problem is also attractive and intractable. We refer to \cite{Dragnev} for a glimpse of this problem.

In \cite{SaffHardinBoroda}, Borodachov, Hardin, and Saff investigated asymptotic properties of minimal $N$-point $s$-energy constants and configurations for fixed $N$ and varying $s.$ Because this will be our main interest in this paper, we will state these results below.

The first theorem \cite[Theorems 2.7.1 and Theorem 2.7.3]{SaffHardinBoroda} concerns the continuity and differentiability of the function 
\begin{equation}\label{eq1}
f(s):=\mathcal{E}^s(A,N),\quad s\geq 0.
\end{equation}
In order to state such theorem, let us define a set
\begin{equation}\label{Gfunction}
G^s_{\log}(A,N):=\bigg\{\sum_{\substack{ x\not=y  \\ x,y \in\omega_N}}K^s(x,y)K_{\log}(x,y): \textup{$\omega_N\subset A$ and $E^{s}(\omega_N)=\mathcal{E}^s(A,N)$}\bigg\},
\end{equation}
for $s\geq 0$

\noindent \textbf{Theorem A.} \emph{
Let $(A,d)$ be an infinite compact metric space and let $N\geq 2$ be fixed.
Then, 
\begin{enumerate}
\item [(a)] the function $f(s)$ defined in \eqref{eq1} is continuous on $[0,\infty).$
\item [(b)] the function $f(s)$ is right differentiable on $[0,\infty)$ and left differentiable on $(0,\infty)$ with 
$$f'_{+}(s):=\lim_{r\rightarrow s^{+}} \frac{f(r)-f(s)}{r-s}=\inf G^{s}(A,N),\quad s\geq 0,$$
and
$$f'_{-}(s):=\lim_{r\rightarrow s^-} \frac{f(r)-f(s)}{r-s}=\sup G^s(A,N),\quad s>0.$$
\end{enumerate}
}

We will see in Theorems B and C below that there are certain relations between minimal $s$-energy problems, as $s\rightarrow 0^{+},$ and best-packing problem defined as follows. The \emph{$N$-point best-packing distance of the set $A$} is defined
\begin{equation}\label{eqjsdfn}
\delta_N(A):=\max \{\delta(\omega_N): \omega_N \subset A \},
\end{equation}
where 
$$\delta(\omega_N):=\min_{1\leq i\not=j\leq N} d(x_i,x_j)$$
denotes the \emph{separation distance of an $N$-point configuration} $\omega_N = \{x_1, \ldots , x_N\},$
and \emph{$N$-point best-packing configurations} are $N$-point configurations attaining the maximum in \eqref{eqjsdfn}.

The following theorem  \cite[Corollary 2.7.5 and Proposition 3.1.2]{SaffHardinBoroda} explains the behavior of $\mathcal{E}^s(A,N)$ as $s\rightarrow 0^{+}$ and $s\rightarrow \infty$. 

\noindent \textbf{Theorem B.} \emph{
For $N\geq 2$ and an infinite compact metric space $(A,d),$
$$\lim_{s\rightarrow 0^{+}} \frac{\mathcal{E}^s(A,N)-N(N-1)}{s}=\mathcal{E}_{\log}(A,N)$$
and
$$\lim_{s \rightarrow \infty} \left(\mathcal{E}^s(A,N)\right)^{1/s}=\frac{1}{\delta_N(A)}.$$
}

Before we state more results, let us define a cluster configuration.
Let $s_0\in [0,\infty]$ We say that 
\begin{itemize}
\item an $N$-point configuration $\omega_N\subset A$ is a \emph{cluster configuration of $\omega_N^{s}$ as $s\rightarrow s_0^{+}$} if there is a sequence $\{s_k\}_{k=1}^{\infty}\subset (s_0,\infty)$ such that $\displaystyle\lim_{k \rightarrow \infty} s_k = s_0$  and $\displaystyle\lim_{k \rightarrow \infty}\omega_N^{s_k}=\omega_N$ in the topology of $A^N$ induced by the metric $d$.
\item an $N$-point configuration $\omega_N\subset A$ is a \emph{cluster configuration of $\omega_N^{s}$ as $s\rightarrow s_0^{-}$}  if there is a sequence $\{s_k\}_{k=1}^{\infty}\subset [0,s_0)$ such that $\displaystyle\lim_{k \rightarrow \infty} s_k = s_0$   and $\displaystyle\lim_{k \rightarrow \infty}\omega_N^{s_k}=\omega_N$ in the topology of $A^N$ induced by the metric $d$.
\item an $N$-point configuration $\omega_N\subset A$ is a \emph{cluster configuration of $\omega_N^{s}$ as $s\rightarrow s_0$}  if there is a sequence $\{s_k\}_{k=1}^{\infty}\subset [0,\infty)$ such that $\displaystyle\lim_{k \rightarrow \infty} s_k = s_0$  and $\displaystyle\lim_{k \rightarrow \infty}\omega_N^{s_k}=\omega_N$ in the topology of $A^N$ induced by the metric $d$.
\end{itemize}

The properties of cluster configurations of minimal $N$-point $s$-energy configurations as $s$ varies (see \cite[Theorem 2.7.1 and Proposition 3.1.2]{SaffHardinBoroda}) are in

\noindent \textbf{Theorem C.} \emph{
Let $(A,d)$ be an infinite compact metric space and, for $s\geq 0$ and $N\geq 2,$ let $\omega_N^s$ denote a minimal $N$-point $s$-energy configuration on $A.$ Then,
\begin{enumerate}
\item [(a)] for $s_0>0,$ any cluster configuration of $\omega_N^{s}$ as $s\rightarrow s_0$ is a minimal $N$-point $s_0$-energy configuration; 
\item [(b)] any cluster configuration of $\omega_N^{s}$ as $s\rightarrow 0^{+}$ is a minimal $N$-point $\log$-energy configuration; 
\item [(c)] any cluster configuration of $\omega_N^{s}$ as $s\rightarrow \infty$ is a $N$-point best-packing configuration.
\end{enumerate}
 }

%The problem is surprisingly difficult and is solved only for few handful values of $N.$ For a glimpse of this problem, we refer to \cite{5points} and references
%there in. The study of distribution of minimal $N$-point $s$-energy configurations as $N\rightarrow \infty$ can be founded in \cite{Landkof,saff,curve}. In 2005, Hardin and Saff showed that if $A$ is an infinite compact subset of the Euclidean space $\mathbb{R}^m$ having positive $n$-dimensional Hausdorff measure, then minimal $N$-point $s$-energy configurations of $A$ are asymptotically uniformly distributed with respect to $n$-dimensional Hausdorff measure on $A$ when $s\geq n$. This result is common known as Poppy-seed bagel theorem.f

% has positive d-dimensional Hausdorff measure,
%minimal 
% $N$-point $s$-energy configurations of an infinite compact subset of 
%
%when $s$ is less than the Hausdorff dimension of $A$ can be founded in 

In this paper, we consider the following \emph{$s,\log^t$-kernel}
\begin{equation}\label{slogker}
K^s_{\log^t}(x,y)=\frac{1}{d(x,y)^s}\left(\log  \frac{1}{d(x,y)}\right)^t,\quad s\geq 0,\quad t\geq 0.
\end{equation}
with corresponding \emph{$s,\log^t$-energy of $\omega_N$} and \emph{minimal $N$-point $s,\log^t$-energy of the set $A$}
$$E^s_{\log^t}(\omega_N):=E_{K^s_{\log^t}}(\omega_N)\quad \textup{and} \quad \quad \mathcal{E}^s_{\log^t}(A,N):=\min_{\substack{\omega_N \subset A  \\ \#\omega_N=N}} E^s_{\log^t}(\omega_N),$$
respectively.
We set
$$\omega_N^{s,\log^t}:=\omega_N^{K_{s,\log^t}},$$
and call it a \emph{minimal $N$-point $s,\log^t$-energy configuration}. Note that the kernel $K^s_{\log^t}(x,y)$ is lower semicontinuous on $A\times A$ and this $s,\log^t$-energy can be viewed as a generalization of both $s$-energy and $\log$-energy. The kernel in \eqref{slogker} was first appeared in the study of the differentiability of the function $f(s)$ in \cite[Theorem 2.7.3]{SaffHardinBoroda}. To the authors' knowledge, no study involving $s,\log^t$-energy constants and configurations appears in the literature.

The main goal of this paper is to prove analogues of Theorems A, B, and C for $s,\log^t$-energy constants and configurations. We would like to emphasize that we will limit our interest to the sets $A$ with $\textup{diam}(A)<1,$ where 
$$\textup{diam}(A):=\sup_{x,y\in A} d(x,y)$$
denotes the diameter of $A.$
For the cases where $\textup{diam}(A)\geq 1,$ the values of the kernel $K^s_{\log^t}(x,y)$ can be $0$ or negative and the analysis becomes laborious. 
Furthermore, we investigate the arrangement of $\omega_N^{s,\log^t}$ on circles in $\mathbb{R}^2$ for certain values of $s$ and $t.$ 

An outline of this paper is as follows. The main results in this paper are stated
in Section \ref{mainresults1}. We keep all auxiliary lemmas in Section \ref{123}.
The proofs of the main results are in Section \ref{ProofofEV}.

\section{Main Results}\label{mainresults1} 

Asymptotic behavior of minimal $N$-point $s,\log^t$-energy constants and configurations as $s\rightarrow\infty$ can be explained in the following theorem.
\begin{thm}\label{The2}
Let $N \geq 2$ and $t \geq 0$ be fixed. Assume that $(A,d)$ is an infinite compact metric space with $\textup{diam}(A)<1$. 
Then, 
\begin{equation*}
\lim_{s \to \infty} \left( \mathcal{E}^{s}_{\log^t} (A,N) \right)^{1/s}=\frac{1}{\delta_N(A)}.
\end{equation*}
Furthermore, every cluster configuration of $\omega_N^{s,\log^t}$ as $s \to \infty$ is an $N$-point best-packing configuration on $A$.
\end{thm}

For a fixed $t\geq 0,$ we define
$$g(s):=\mathcal{E}^s_{\log^t} (A,N),\quad \quad s\geq 0.$$
The continuity of $g(s)$ is stated below.
\begin{thm}\label{contE}
Let $N \geq 2$ and $t \geq 0$ be fixed. Assume that $(A,d)$ is an infinite compact metric space with $\textup{diam}(A)<1.$  Then, the function $g(s)$ is continuous on $[0,\infty)$.
\end{thm}

Analysis of cluster configurations of $\omega_N^{s,\log^t}$ as $s\rightarrow s_0>0$ is in the following theorem.
\begin{thm}\label{clusters0}
Let $N \geq 2$ and $t \geq 0$ be fixed. Assume that $(A,d)$ is an infinite compact metric space with $\textup{diam}(A)<1$. 
Denote by $\omega_N^{s,\log^{t}}$ a minimal $N$-point $s,\log^{t}$-energy configuration on $A$. Then, for any $s_{0}>0,$ any cluster configuration of $\omega_N^{s,\log^{t}},$ as $s \to s_0,$ is a minimal $N$-point $s_0,\log^{t}$-energy configuration on $A$.
\end{thm}

For $s\geq 0$ and $t\geq 0,$ we set
$$G^s_{\log^{t+1}}(A,N):=\{E^{s}_{\log^{t+1}}(\omega_N): \textup{$\omega_N\subset A$ \textup{and} $E^{s}_{\log^{t}}(\omega_N)=\mathcal{E}^s_{\log^t}(A,N)$}\}.$$

The differentiability properties of $g(s)$ are in Theorems \ref{tptn} and \ref{diff}.

\begin{thm}\label{tptn}
Let  $N \geq 2$ and $t\geq 0$ be fixed. Assume that $(A,d)$ is an infinite compact metric space with $\textup{diam}(A)<1.$ Then, the function $g(s)$ is right differentiable on $[0,\infty)$ and left differentiable on $(0,\infty)$ with 
\begin{equation}\label{tptn1}
g'_{+}(s):=\lim_{r\rightarrow s^{+}} \frac{g(r)-g(s)}{r-s}=\inf G^{s}_{\log^{t+1}}(A,N),\quad s\geq 0,
\end{equation}
and
\begin{equation}\label{tptn2}
g'_{-}(s):=\lim_{r\rightarrow s^-} \frac{g(r)-g(s)}{r-s}=\sup G^s_{\log^{t+1}}(A,N),\quad s>0.
\end{equation}
\end{thm}
\begin{thm}\label{diff} Let  $N \geq 2$ and $t\geq 0$ be fixed. Assume that $(A,d)$ is an infinite compact metric space with $\textup{diam}(A)<1.$ Then,
\begin{enumerate}
\item [(a)]
 the function $g(s)$ is differentiable at $s=s_0>0$ if and only if 
$$\inf G^{s_0}_{\log^t}(A,N) = \sup G^{s_0}_{\log^t}(A,N);$$
\item [(b)] if $\omega_N^{*}$ is a cluster point of $ \omega_N^{s,\log^{t}} $ as $s\rightarrow s_{0}^{+}\geq 0$, then
$$ E^{s_0}_{\log^{t+1}}(\omega_N^{*})=\inf G^{s_0}_{\log^{t+1}}(A,N)=g'_{+}(s_0);$$
\item [(c)] if $\omega_N^{**}$ is a cluster point of $ \omega_N^{s,\log^{t}} $ as $s\rightarrow s_{0}^{-}> 0$, then
$$ E^{s_0}_{\log^{t+1}}(\omega_N^{**})=\sup G^{s_0}_{\log^{t+1}}(A,N)=g'_{-}(s_0);$$
\item [(d)] for $s_0>0,$ if there exists a configuration $\omega_N^{*}$ that is both cluster configurations of  $ \omega_N^{s,\log^{t}} $ as $s\rightarrow s_{0}^{+}$ and $s\rightarrow s_{0}^{-},$ then the function $g(s)$ is differentiable at $s=s_0$
with 
$$E^{s_0}_{\log^{t+1}}(\omega_N^{*})=g'(s_0).$$
\end{enumerate}
\end{thm}

 Let $d_u$ be the $2$-dimensional Euclidean metric of $\mathbb{R}^2$. For $\alpha>0,$ we denote by
$$\mathbb{S}^1_{\alpha}:=\{x\in \mathbb{R}^2: d_u(0,x)=\alpha\}$$
the circle centered at $0$ of radius $\alpha.$ 
We let $L(x, y)$ be the geodesic distance between the points $x$ and $y$ on $\mathbb{S}^1_{\alpha}$; that is, the length of the shorter arc of $\mathbb{S}^1_{\alpha}$ connecting the points $x$ and $y.$  

The optimality of $N$ distinct equally spaced points on $\mathbb{S}^1_{\alpha}$ with the Euclidean metric $d_u$ or the geodesic distance $L$ for the certain $s,\log^t$-energy problems is stated in Propositions \ref{Equalspace1}-\ref{Equalspace3}.

\begin{Proposition}\label{Equalspace1} 
Let  $N \geq 2,$ $s\geq 0,$ $t\geq 1,$ and $0<\alpha<\pi^{-1}.$ Then, $\omega_N$ is a minimal $N$-point $s,\log^t$-energy configuration on $\mathbb{S}^1_{\alpha}$ with the geodesic distance $L$ if and only if $\omega_N$ is a configuration of $N$ distinct equally spaced points on $\mathbb{S}^1_{\alpha}$.
\end{Proposition}

\begin{Proposition}\label{Equalspace2}
Let  $N \geq 2,$  $0<\alpha<(e\pi)^{-1},$ and $s,t$ satisfy $s> 0,t\geq 0$ or $s=0,t>0.$
Then, $\omega_N$ is a minimal $N$-point $s,\log^t$-energy configuration on $\mathbb{S}^1_{\alpha}$ with  the geodesic distance $L$ if and only if $\omega_N$ is a configuration of $N$ distinct equally spaced points on $\mathbb{S}^1_{\alpha}$.
\end{Proposition}

\begin{Proposition}\label{Equalspace3} Let  $N \geq 2,$ $s\geq  0,$ $t\geq 1,$ and $0<\alpha<1/2.$  Then, $\omega_N$ is a minimal $N$-point $s,\log^t$-energy configuration on $\mathbb{S}^1_{\alpha}$ with  the Euclidean metric $d_u$ if and only if $\omega_N$ is a configuration of $N$ distinct equally spaced points on $\mathbb{S}^1_{\alpha}$.
\end{Proposition}

Note that the conditions $0<\alpha<\pi^{-1}$ in Proposition \ref{Equalspace1} and $0<\alpha<1/2$ in Proposition \ref{Equalspace3} are needed to make sure that $\textup{diam}(\mathbb{S}^1_{\alpha})<1$ corresponding to the Euclidean metric $d_u$ and the geodesic distance $L,$ respectively.

\section{Auxiliary Lemmas}\label{123}

\begin{Lemma}\label{inc}
Let $\beta \geq 0$ and $h:(0,1) \to (0,\infty)$ be a function defined by
$$ h(x):=x\left( \log \frac{1}{x} \right)^{-\beta}\quad \textup{for all $x\in (0,1)$}.$$
Then, $h(x)$ is strictly increasing on $(0,1)$.
\end{Lemma}

\begin{proof}[Proof of Lemma \ref{inc}]
Because
$$h'(x)=\beta \left( \log \frac{1}{x} \right)^{-(\beta+1)}+\left( \log \frac{1}{x} \right)^{-\beta}$$
and $( \log (1/x) )^{-\beta}>0$ for all $x \in (0,1)$ and $ \beta \geq0$, $h'(x)>0$ for all $x \in (0,1)$. 
Therefore, $h(x)$ is strictly increasing on $(0,1)$.
\end{proof}

\begin{Lemma}\label{decre}
Let $(s,t)\in [0,\infty)\times [0,\infty)]\setminus\{(0,0)\}$ and $p:(0,1) \to (0,\infty)$ be a function defined by
$$ p(x):=\frac{1}{x^s}\left( \log \frac{1}{x} \right)^{t}\quad \textup{for all $x\in (0,1)$}.$$
Then, $p(x)$ is strictly decreasing on $(0,1)$.
\end{Lemma}
\begin{proof}[Proof of Lemma \ref{decre}] Using Lemma \ref{inc}, we set  $\beta=t/s$ and 
 $$p(x)=\left(\frac{1}{h(x)}\right)^s=\frac{1}{x^s}\left( \log \frac{1}{x} \right)^{t}$$
 is strictly decreasing on $(0,1)$.
\end{proof}

\begin{Lemma}\label{BoundedBel} Let $(A,d)$ be an infinite compact metric space with $\textup{diam}(A)<1$ and  $s,t\geq 0.$
Then, for all $N$-point configurations $\omega_N\subset A,$
$$E_{\log^t}^s(\omega_N)\geq \frac{N(N-1)}{(\textup{diam}(A))^s} \left( \log \frac{1}{\textup{diam}(A)}\right)^t.$$
\end{Lemma}
\begin{proof}[Proof of Lemma \ref{BoundedBel}] The proof relies on the fact that $p(x)$ in Lemma \ref{decre}
 is strictly decreasing on $(0,1)$.
\end{proof}

\begin{Lemma}\label{ineqE}
Let $(A,d)$ be an infinite compact metric space with $\textup{diam}(A)<1$ and $\omega_N=\{x_1,\ldots,x_N\}$ be any configuration of $N$ distinct points of $A$.
Then, for any $s>r\geq 0 $ and $t \geq 0$, 
\begin{equation*}
 E^r_{\log^{t+1}}(\omega_N) \leq \frac{E^s_{\log^{t}}(\omega_N)-E^{r}_{\log^{t}}(\omega_N)}{s-r} \leq E^{s}_{\log^{t+1}}(\omega_N).
\end{equation*}
\end{Lemma}
\begin{proof}[Proof of Lemma \ref{ineqE}]
 Let  $x_i,x_j \in \omega_N$ where $1\leq i \not=j \leq N$, let $s>r\geq0$, and let $t \geq 0$.
Then,
\begin{equation*}
\frac{1}{d(x_i,x_j)^r}\log\frac{1}{d(x_i,x_j)} \leq \frac{\displaystyle \frac{1}{d(x_i,x_j)^{s}}-\frac{1}{d(x_i,x_j)^{r}}}{s-r} \leq \frac{1}{d(x_i,x_j)^s}\log \frac{1}{d(x_i,x_j)}.
\end{equation*}
Since $\left(\log \displaystyle \frac{1}{d(x_i,x_j)}\right)^{t} >0$,
$$
\frac{1}{d(x_i,x_j)^r}\left(\log\frac{1}{d(x_i,x_j)} \right)^{t +1}  \leq \frac{\displaystyle \frac{1}{d(x_i,x_j)^{s}}\left(\log \displaystyle \frac{1}{d(x_i,x_j)}\right)^{t} -\frac{1}{d(x_i,x_j)^{r}}\left(\log \displaystyle \frac{1}{d(x_i,x_j)}\right)^{t} }{s-r} $$
$$ \leq \frac{1}{d(x_i,x_j)^s}\left(\log \frac{1}{d(x_i,x_j)}\right)^{t +1}.$$
It follows that 
\begin{equation*}
E^{r}_{\log^{t+1}}(\omega_N) \leq \frac{E^{s}_{\log^{t}}(\omega_N)-E^{r}_{\log^{t}}(\omega_N)}{s-r} \leq E^{s}_{\log^{t+1}}(\omega_N).
\end{equation*}
\end{proof}

Let $\Gamma$ be a rectifiable simple closed curve in $\mathbb{R}^m, m \geq 2,$ of length $|\Gamma|$ with a chosen orientation. We recall that $L(x, y)$ is the geodesic distance between the points $x$ and $y$ on $\Gamma.$ With the help of the following lemma \cite[Theorem 2.3.1]{SaffHardinBoroda}, we can prove Propositions \ref{Equalspace1}-\ref{Equalspace3}.

\begin{Lemma}\label{Optimal}
Let $k : (0, |\Gamma| /2] \rightarrow \mathbb{R}$ be a strictly convex and decreasing function defined at $u = 0$ by the (possibly infinite) value $\displaystyle\lim_{u \rightarrow 0^+} k(u)$ and let $K$ be the kernel on $\Gamma \times \Gamma$ defined by $K(x, y) = k (L(x, y)).$ Then, all  minimal $N$-point $K$-energy configurations on $\Gamma$  are configurations of $N$ distinct equally spaced points on $\Gamma$ with respect  to the arc length and vice versa.
\end{Lemma}

\section{Proofs of Main Results}\label{ProofofEV}

\begin{proof}[ Proof of Theorem \ref{The2}]
Let $t\geq 0$ be fixed, $s>0,$  $\omega_N^{s,\log^t}$ be a minimal $N$-point $s,\log^{t}$-energy configuration on $A$, and  let  $\omega_N^{\infty}$ be an $N$-point best-packing configuration on $A$. 
Since $\textup{diam}(A)<1$ and points in $\omega_N^{s,\log^t}$ are distinct, there is a constant $c>0$ such that
\begin{equation*}
0<\delta(\omega^{s,\log^t}_N) \leq c<1
\end{equation*}
where the constant $c$ depends only on the set $A.$
This implies that 
\begin{equation*}
\left(\log\frac{1}{c}\right)^{t}  \leq \left(\log\frac{1}{\delta(\omega_N^{s,\log^t})}\right)^{t} .
\end{equation*}
Then,
$$
\frac{1}{\delta_N(A)}\left(\log \frac{1}{c}\right)^{t/s} \leq \frac{1}{\delta(\omega_N^{s,\log^t})}\left(\log \frac{1}{c}\right)^{t/s} \leq \frac{1}{\delta(\omega_N^{s,\log^t})}\left(\log\frac{1}{\delta(\omega_N^{s,\log^t})}\right)^{t/s}$$
\begin{equation}\label{simi}
\leq
\left(E^{s}_{\log^{t}}(\omega_N^{s,\log^t})\right)^{1/s}=\left(\mathcal{E}^s_{\log^{t}}(A,N)\right)^{1/s}  \leq \left(E^{s}_{\log^{t}}(\omega_N^{\infty})\right)^{1/s} \leq \frac{1}{\delta_N(A)}\left( E_{\log^t} (\omega_N^{\infty}) \right)^{1/s}.
\end{equation}
Since 
$$\lim_{s \to \infty} \frac{1}{\delta_N(A)} \left( \log \frac{1}{c}\right)^{t/s}=\frac{1}{\delta_N(A)}$$
and
$$\lim_{s \to \infty} \frac{1}{\delta_N(A)}\left(  E_{\log^t} (\omega_N^{\infty}) \right)^{1/s}=\frac{1}{\delta_N(A)},$$
it follows that
$$ \lim_{s \to \infty} \left(\mathcal{E}^{s}_{\log^{t}} (A,N)\right)^{1/s}=\frac{1}{\delta_N(A)}.$$

Let $\omega_N^*$ be a cluster configuration of $\omega_N^{s,\log^t}$ as $s \to \infty$. 
 This implies that there is a sequence $\{s_k\}_{k=1}^{\infty} \subset \mathbb{R}$ such that $s_k \rightarrow \infty$ and $\omega_N^{s_k,\log^t} \to \omega_N^{*}$ as $k \to \infty$.
Arguing as in \eqref{simi}, we have
$$
 \frac{1}{\delta(\omega_N^{s_k,\log^t})}\left(\log\frac{1}{c}\right)^{t/s_k}\leq
\left(E^{s_k}_{\log^{t}}(\omega_N^{s_k,\log^t})\right)^{1/s_k}=\left(\mathcal{E}^{s_k}_{\log^{t}}(A,N)\right)^{1/s_k} \leq  \left(E^{s_k}_{\log^{t}}(\omega_N^{\infty})\right)^{1/s_k}
$$
$$\leq \frac{1}{\delta(\omega_N^{\infty})}\left( E_{\log^{t}} (\omega_N^{\infty})\right)^{1/s_k}.
$$
Taking $k \to \infty$, we obtain
$$\delta_N(A)=\delta(\omega_N^{\infty})\leq \delta(\omega_N^{*}).$$
This means that $\omega_N^{*}$ is also an $N$-point best-packing configuration on $A$.
\end{proof}

\begin{proof}[Proof of Theorem \ref{contE}] First of all, we show that $g(s)$ is continuous on $(0,\infty).$ Let $s>0$ and let $\omega_N^{s,\log^{t}}$ be a minimal $N$-point $s,\log^{t}$-energy configuration on $A$. Using Lemma \ref{ineqE}, we obtain for any $\omega_N^{s,\log^t},$
$$
\liminf_{r \to s^{-}} \frac{g(r)-g(s)}{r-s} \geq  \liminf_{r \to s^{-}} \frac{E^{r}_{\log^{t}} (\omega_N^{s,\log^{t}})-E^{s}_{\log^{t}} (\omega_N^{s,\log^{t}})}{r-s}
$$
\begin{equation}\label{inftneg}
\geq \lim_{r \to s^{-}} E^{r}_{\log^{t+1}} (\omega_N^{s,\log^{t}}) =E^{s}_{\log^{t+1}} (\omega_N^{s,\log^{t}}) \geq \sup G^{s}_{\log^{t+1}} (A,N) >0,
\end{equation}
 and 
\begin{align}\label{suptneg1}
\limsup_{r \to s^{-}}  \frac{g(r)-g(s)}{r-s} 
\leq \limsup_{r \to s^{-}} \frac{E^{r}_{\log^{t}} (\omega_N^{r,\log^{t}})-E^{s}_{\log^{t}} (\omega_N^{r,\log^{t}})}{r-s} \leq \limsup_{r \to s^{-}}E^{s}_{\log^{t+1}} (\omega_N^{r,\log^{t}}),
\end{align}
where the second inequality in \eqref{inftneg} follows from the arbitrariness of $\omega_N^{s,\log^t}$ and the last inequality in \eqref{inftneg} follows from Lemma \ref{BoundedBel}.

Let $\omega_N$ be a fixed configuration of $N$ distinct points of $A$.
Note that $0< \delta(\omega_N) <1$.
For all $r\in (s/2,s)$, we have
$$
\left(\frac{1}{\delta(\omega_N^{r,\log^{t}})}\right)^{s/2}\left(\log \frac{1}{\delta(\omega_N^{r,\log^{t}})} \right)^{t} \leq \left( \frac{1}{\delta(\omega_N^{r,\log^{t}})}\right)^{r}\left( \log \frac{1}{\delta(\omega_N^{r,\log^{t}})} \right)^{t} 
\leq E^{r}_{\log^{t}}(\omega_N^{r,\log^{t}})$$
$$
\leq  E^{r}_{\log^{t}}(\omega_N) \leq \left( \frac{1}{\delta(\omega_N)} \right)^{r}\left(\log \frac{1}{\delta(\omega_N)} \right)^{t} N(N-1)
$$
$$\leq 
\left( \frac{1}{\delta(\omega_N)}\right)^{s}\left(\log \frac{1}{\delta(\omega_N)} \right)^{t} N(N-1).
$$
That is,
$$(\delta(\omega_N^{r,\log^{t}}))^{s/2}\left(\log \frac{1}{\delta(\omega_N^{r,\log^{t}})} \right)^{-t} \geq ( \delta(\omega_N))^{s}\left(\log \frac{1}{\delta(\omega_N)} \right)^{-t}( N(N-1))^{-1}.$$
This implies that for all $r\in (s/2,s),$
$$\delta(\omega_N^{r,\log^{t}})\left(\log \frac{1}{\delta(\omega_N^{r,\log^{t}})} \right)^{-2t/s} \geq ( \delta(\omega_N))^{2}\left(\log \frac{1}{\delta(\omega_N)} \right)^{-2t/s}( N(N-1))^{-2/s}=:c_1>0.$$
Since by Lemma \ref{inc}, $$h(x):=x\left(\log \frac{1}{x}\right)^{-\beta},\quad  \beta>0,$$ is a strictly increasing function on $(0,1)$, there exists a constant $c_2>0$ such that for all $r\in (s/2,s),$
$$ \delta(\omega_N^{r,\log^{t}}) \geq c_2 >0.$$
Therefore, $ E^{s}_{\log^{t+1}}(\omega_N^{r,\log^{t}})$ are bounded above where $r\in (s/2,s)$. 
From this and \eqref{suptneg1}, 
\begin{equation}\label{suptneg}
\limsup_{r \to s^{-}}  \frac{g(r)-g(s)}{r-s} 
\leq  \limsup_{r \to s^{-}}E^{s}_{\log^{t+1}} (\omega_N^{r,\log^{t}})<\infty.
\end{equation}

Let $s\geq 0.$ Using Lemma \ref{ineqE}, we also obtain for any $\omega_N^{s,\log^t},$
$$
\limsup_{r \to s^{+}}  \frac{g(r)-g(s)}{r-s}
\leq \limsup_{r \to s^{+}} \frac{E^{r}_{\log^{t}} (\omega_N^{s,\log^{t}})-E^{s}_{\log^t} (\omega_N^{s,\log^{t}})}{r-s}
$$
\begin{equation}\label{suptpos}
\leq \lim_{r \to s^{+}}E^{r}_{\log^{t+1}} (\omega_N^{s,\log^{t}})=  E^{s}_{\log^{t+1}}(\omega_N^{s,\log^{t}})
\leq \inf G^{s}_{\log^{t+1}}(A,N)<\infty,
\end{equation}
and
\begin{align}\label{inftpos}
\liminf_{r \to s^{+}} \frac{g(r)-g(s)}{r-s} 
 \geq \liminf_{r \to s^{+}} \frac{E^{r}_{\log^{t}} (\omega_N^{r,\log^{t}})-E^{s}_{\log^{t}} (\omega_N^{r,\log^{t}})}{r-s}
 \geq \liminf_{r \to s^{+}} E^{s}_{\log^{t+1}}(\omega_N^{r,\log^{t}})>0,
\end{align}
where the second inequality in \eqref{suptpos} follows from rom the arbitrariness of $\omega_N^{s,\log^t}$ and the last inequality in \eqref{inftpos} follows from Lemma \ref{BoundedBel}.

The inequalities (\ref{inftneg}), (\ref{suptneg}), (\ref{suptpos}), and (\ref{inftpos}) imply that for all $s>0,$
\begin{equation}\label{leftcon}
0<\liminf_{r \to s^{-}} \frac{g(r)-g(s)}{r-s}  \leq \limsup_{r \to s^{-}}  \frac{g(r)-g(s)}{r-s} <\infty
\end{equation}
and for all $s\geq 0$
\begin{equation}\label{rightcon}
0<\liminf_{r \to s^{+}} \frac{g(r)-g(s)}{r-s} \leq \limsup_{r \to s^{+}}  \frac{g(r)-g(s)}{r-s} <\infty.
\end{equation}
The inequalities in \eqref{leftcon} and \eqref{rightcon} further imply that $g(s)$ is continuous for all $s> 0$ and is right continuous at $s=0$.
\end{proof}

\begin{proof}[Proof of Theorem \ref{clusters0}]
Let $s_0 >0.$ In order to show Theorem \ref{clusters0}, it suffices to show that any cluster configuration of $\omega_N^{s,\log^{t}}$ as $s \to s_0^{+}$ or as $s \to s_0^{-}$ is a minimal $N$-point $s_0,\log^{t}$-energy configuration on $A$.   

 Let $\omega_N^*$ be a cluster configuration of $\omega_N^{s,\log^{t}},$ as $s \to s_0^{+}.$ 
Then,
there is a sequence $\{ s_k\}_{k=1}^{\infty}\subset (s_o,\infty)$ such that $s_k \to s_0$ and  $\omega_N^{s_k,\log^{t}} \to  \omega_N^*$ as $k \to \infty$.
Let $\alpha=\textup{diam}(A).$ For any configuration of $N$ distinct points $\omega_N$ on $A$,
notice that
$\alpha^{s} E^{s}_{\log^{t}}(\omega_N)$ is an increasing function of $s$.
Applying the continuity of  $g(s):=\mathcal{E}_{\log^t}^s(A,N)$ at $s_0$, we have
$$
\alpha^{s_0}E^{s_0}_{\log^{t}}(\omega_N^{*})=\lim_{k \to \infty} \alpha^{s_0} E^{s_0}_{\log^{t}}( \omega_N^{s_k,\log^{t}}) \leq \lim_{k \to \infty} \alpha^{s_k} E^{s_k}_{\log^{t}}( \omega_N^{s_k,\log^{t}})
$$
$$=\lim_{k \to \infty} \alpha^{s_k} \mathcal{E}_{\log^t}^{s_k}(A,N)=\alpha^{s_0} \mathcal{E}_{\log^t}^{s_0}(A,N).
$$
This implies that $E^{s_0}_{\log^{t}}(\omega_N^{*})=\mathcal{E}_{\log^t}^{s_0}(A,N)$.
Hence, $\omega_N^{*}$ is a minimal $N$-point $s_0,\log^{t}$-energy configuration on $A$.

Let $\omega_N^{**}$ be a cluster configuration of $\omega_N^{s,\log^{t}},$ as $s \to s_0^{-}.$
Then, there is a sequence $\{ s_k\}_{k=1}^{\infty} \subset [0,s_0)$ such that $s_k \to s_0$ and
 $\omega_N^{s_k,\log^{t}} \to \omega_N^{**}$ as $k \to \infty$. 
Without loss of generality, we may assume that $s_0 /2<s_k<s_0$ for all $k$.
For any configuration of $N$ distinct points $\omega_N$ of $A$,
observe that
$\delta(\omega_N)^{s}E^{s}_{\log^{t}}(\omega_N)$ is a decreasing function of $s$.
It follows from the continuity of the function $g(s)$ that $g(s)$ is bounded above by some number $M>1$ for all $s \in (s_0/2,s_0).$
For every $s_0/2 < s_k < s_0$,
$$
(\delta(\omega_N^{s_k,\log^{t}}))^{-s_0/2}\bigg(\log \frac{1}{\delta(\omega_N^{s_k,\log^{t}})} \bigg)^{t} \leq ( \delta(\omega_N^{s_k,\log^{t}}))^{-s_k}\bigg( \log \frac{1}{\delta(\omega_N^{s_k,\log^{t}})} \bigg)^{t} 
$$
$$
\leq E^{s_k}_{\log^{t}}(\omega_N^{s_k,\log^{t}})\leq M.
$$
Then,
$$\delta(\omega_N^{s_k,\log^{t}})\bigg(\log \frac{1}{\delta(\omega_N^{s_k,\log^{t}})} \bigg)^{-2t/s_0} \geq M^{-2/s_0}>0.$$
Using Lemma \ref{inc}, there is a constant $c>0$ such that
\begin{equation*}
\delta(\omega_N^{s_k,\log^{t}}) \geq c >0 \quad  \text{ for all } \quad k \in \mathbb{N}.
\end{equation*}
Using the continuity of  $g(s):=\mathcal{E}_{\log^t}^s(A,N)$ at $s_0$, we have
$$
(\delta(\omega_N^{**}))^{s_0}  E^{s_0}_{\log^{t}}(\omega_N^{**})= \lim_{k \to \infty} (\delta(\omega_N^{s_k,\log^{t}}))^{s_0}  E^{s_0}_{\log^{t}}(\omega_N^{s_k,\log^{t}})
$$
$$\leq \lim_{k \to \infty} (\delta(\omega_N^{s_k,\log^{t}}))^{s_k}  E^{s_k}_{\log^{t}}(\omega_N^{s_k,\log^{t}})= \lim_{k \to \infty} (\delta(\omega_N^{s_k,\log^{t}}))^{s_k}  \mathcal{E}_{\log^t}^{s_k}(A,N)
$$
$$= (\delta(\omega_N^{**}))^{s_0}  \mathcal{E}_{\log^t}^{s_0}(A,N).
$$
This implies that $ E^{s_0}_{\log^{t}}(\omega_N^{**})= \mathcal{E}_{\log^t}^{s_0}(A,N)$.
Hence, $\omega_N^{**}$ is a minimal $N$-point $s_0,\log^{t}$-energy configuration on $A$.
\end{proof}

\begin{proof}[Proof of Theorem \ref{tptn}]
Firstly, we show \eqref{tptn1}.
Let $s\geq 0$ be fixed and $\{ r_k \}_{k=1}^{\infty} \subset (s,\infty)$ be a sequence  such that $r_k \to s$ as $k \to \infty$ and
\begin{equation}\label{use111}
\lim_{k \to \infty} E^{s}_{\log^{t+1}}(\omega_N^{r_k,\log^{t}})=\liminf_{r \to s^{+}}  E^{s}_{\log^{t+1}}(\omega_N^{r,\log^{t}}).
\end{equation}
Since $A^N$ is compact, there exists a subsequence $\{ s_\ell\}_{\ell =1}^{\infty} \subset \{r_k\}_{k=1}^{\infty}$ such that
\begin{equation}\label{wertyuil}
\lim_{\ell \to \infty}  \omega_N^{s_{\ell},\log^{t}}=\omega_N^{*}
\end{equation}
and $\omega_N^{*}$ is a minimal $N$-point $s,\log^t$-energy configuration by Theorem \ref{clusters0}.
By
$$\lim_{k \to \infty} E^{s}_{\log^{t+1}}(\omega_N^{r_k,\log^{t}})=\lim_{\ell \to \infty} E^{s}_{\log^{t+1}}(\omega_N^{s_\ell,\log^{t}}),$$
\eqref{suptpos}, \eqref{inftpos}, \eqref{use111}, and \eqref{wertyuil}, we get
$$
\liminf_{r \to s^{+}} \frac{g(r)-g(s)}{r-s} 
 \geq \liminf_{r \to s^{+}} E^{s}_{\log^{t+1}}(\omega_N^{r,\log^{t}})
=\lim_{\ell \to \infty} E^{s}_{\log^{t+1}}(\omega_N^{s_\ell,\log^{t}})
$$
\begin{equation}\label{argueas1}
=E^{s}_{\log^{t+1}}(\omega_N^{*})
\geq \inf G^{s}_{\log^{t+1}}(A,N) \\
 \geq \limsup_{r \to s^{+}}  \frac{g(r)-g(s)}{r-s}.
\end{equation}
Then,
\begin{equation}\label{infder}
g'_{+}(s)=\inf G^{s}_{\log^{t+1}}(A,N).
\end{equation}
It is easy to check that from Lemma \ref{BoundedBel}, the constant $\inf G^{s}_{\log^{t+1}}(A,N)$ in (\ref{infder})  is finite. This verifies \eqref{tptn1}.

Next, we prove \eqref{tptn2}.
Let $s>0$ be fixed and $\{ r_k \}_{k =1}^{\infty}\subset [0,s)$ be a sequence such that $r_k \to s$ as $k \to \infty$ and
\begin{equation}\label{use112}
\lim_{k \to \infty} E^{s}_{\log^{t+1}}(\omega_N^{r_k,\log^{t}})=\limsup_{r \to s^{-}}  E^{s}_{\log^{t+1}}(\omega_N^{r,\log^{t}}).
\end{equation}
Because $A^N$ is compact, there exists a subsequence $\{ s_\ell\}_{\ell =1}^{\infty} \subset \{r_k\}_{k=1}^{\infty}$ such that
$$\lim_{\ell \to \infty}  \omega_N^{s_{\ell},\log^{t}}=\omega_N^{**}$$
and $\omega_N^{**}$ is a minimal $N$-point $s,\log^t$-energy configuration by Theorem \ref{clusters0}.
Then, we get
\begin{equation}\label{ewrtsacd}
\lim_{k \to \infty} E^{s}_{\log^{t+1}}(\omega_N^{r_k,\log^{t}})=\lim_{\ell \to \infty} E^{s}_{\log^{t+1}}(\omega_N^{s_\ell,\log^{t}}).
\end{equation}
Using \eqref{inftneg}, \eqref{suptneg1}, \eqref{use112}, and \eqref{ewrtsacd}, we obtain
$$\liminf_{r \to s^{-}} \frac{g(r)-g(s)}{r-s}  
\geq \sup G^{s}_{\log^{t+1}}(A,N)\geq E^{s}_{\log^{t+1}} (\omega_N^{**})
$$
$$= \lim_{\ell \to \infty} E^{s}_{\log^{t+1}} (\omega_N^{s_\ell,\log^{t}}) = \limsup_{ r \to s^{-}}  E^{s}_{\log^{t+1}} (\omega_N^{r,\log^{t}})  
\geq \limsup_{r \to t^{-}} \frac{g(r)-g(s)}{r-s}.
$$
Then,
\begin{equation}\label{supder}
g'_{-}(s)=\sup G^{s}_{\log^{t+1}}(A,N).
\end{equation}

Next, we want to show that $\sup G^{s}_{\log^{t+1}}(A,N)$ is finite.
Let $\omega_N$ be a fixed configuration of $N$ distinct points on $A$ and let $\omega_N^{s,\log^{t}}$ be any minimal $N$-point $s,\log^t$ configurations. Then,
$$
(\delta(\omega_N^{s,\log^{t}}))^{-s}\left(\log \frac{1}{\delta(\omega_N^{s,\log^{t}})} \right)^{t} \leq  E^{s}_{\log^{t}}(\omega_N^{s,\log^{t}})
$$
$$\leq  E^{s}_{\log^{t}}(\omega_N) 
\leq ( \delta(\omega_N))^{-s}\left(\log \frac{1}{\delta(\omega_N)} \right)^{t} N(N-1).
$$
That is,
\begin{equation*}
\delta(\omega_N^{s,\log^{t}})\left(\log \frac{1}{\delta(\omega_N^{s,\log^{t}})} \right)^{-t/s} \geq  \delta(\omega_N)\left(\log \frac{1}{\delta(\omega_N)} \right)^{-t/s}( N(N-1))^{-1/s}=:c_1>0.
\end{equation*}
It follows from Lemma \ref{inc} that there is a constant  $c_2>0$ such that for any $\omega_N^{s,\log^{t}},$
$$\delta(\omega_N^{s,\log^{t}}) \geq c_2 >0.$$
Since by Lemma \ref{decre}, $$p(x):=\frac{1}{x^s}\left( \log \frac{1}{x} \right)^{t+1},$$ is a strictly decreasing function on $(0,1)$, the set $ G^{s}_{\log^{t+1}}(A,N)$ is bounded above.
This implies that $\sup G^{s}_{\log^{t+1}}(A,N)$ in (\ref{supder}) is finite. Hence, \eqref{tptn2} is proved.
\end{proof}

\begin{proof}[Proof of Theorem \ref{diff}]

\textbf{(a):} This is a direct consequence of Theorem \ref{tptn}.

\noindent \textbf{(b):} Let $s_0\geq 0$ and $\omega_N^{*}$ be a cluster configuration of $\{ \omega_N^{s,\log^{t}} \}$ as $s\rightarrow s_{0}^{+}$.  Then, there exists a sequence $\{s_k\}_{k=1}^{\infty} \subset (s_0,\infty)$ such that $\displaystyle\lim_{k\rightarrow \infty}s_k = s_0$ and $\displaystyle\lim_{k \to \infty} \omega_N^{s_k,\log^{t}}=\omega_N^{*}$.
Then, $\omega_N^{*}$ is a minimal $N$-point $s_0,\log^t$-energy configuration by Theorem \ref{clusters0}. Using  \eqref{tptn1} and the similar argument used to show \eqref{inftpos}, we have
$$
 E^{s_0}_{\log^{t+1}}(\omega_N^{*})=\lim_{k \to \infty} E^{s_0}_{\log^{t+1}}(\omega_N^{s_k,\log^{t}})
 \leq \lim_{k \to \infty } \frac{g(s_k)-g(s_0)}{s_k-s_0} =g'_{+}(s_0)=\inf G^{s_0}_{\log^{t+1}}(A,N).
$$
Since $\inf G^{s_0}_{\log^{t+1}}(A,N) \leq  E^{s_0}_{\log^{t+1}}(\omega_N^{*}),$
$$  E^{s_0}_{\log^{t+1}}(\omega_N^{*})=\inf G^{s_0}_{\log^{t+1}}(A,N)=g'_{+}(s_0).$$

\noindent \textbf{(c):} Let $s_0> 0$ and $\omega_N^{**}$ be a cluster configuration of $\{ \omega_N^{s,\log^{t}} \}$ as $s\rightarrow s_{0}^{-}$.  Then, there exists a sequence $\{s_k\}_{k=1}^{\infty} \subset [0,s_0)$ such that $\displaystyle \lim_{k\rightarrow \infty}s_k = s_0$ and $\displaystyle\lim_{k \to \infty} \omega_N^{s_k,\log^{t}}=\omega_N^{**}$.
Then, $\omega_N^{**}$ is a minimal $N$-point $s_0,\log^t$-energy configuration by Theorem \ref{clusters0}. Using  \eqref{tptn2} and  the similar argument used to show \eqref{suptneg}, we have
$$
E^{s_0}_{\log^{t+1}}(\omega_N^{**})=\lim_{k \to \infty} E^{s_0}_{\log^{t+1}}(\omega_N^{s_k,\log^{t}}) \geq   \lim_{k \to \infty}  \frac{g(s_k)-g(s_0)}{s_k-s_0}= g'_{-}(s_0)= \sup G^{s_0}_{\log^{t+1}} (A,N).
$$
Since $E^{s_0}_{\log^{t+1}}(\omega_N^{**}) \leq \sup G^{s_0}_{\log^{t+1}}(A,N)$,
$$E^{s_0}_{\log^{t+1}}(\omega_N^{**}) = \sup G^{s_0}_{\log^{t+1}} (A,N)=g'_{-}(s_0).$$

\noindent \textbf{(d):} This is a direct consequence of (b) and (c).
\end{proof}

\begin{proof}[Proof of Proposition \ref{Equalspace1}] Let  $N \geq 2,$ $s\geq 0,$ $t\geq 1,$ and $0<\alpha<\pi^{-1}.$ We prove this proposition using Lemma \ref{Optimal}.
The function $k:(0,1):\rightarrow \mathbb{R}$ in the lemma is 
$$k(x)=\frac{1}{x^s}\left( \log\frac{1}{x}\right)^t.$$
By Lemma \ref{decre}, $k(x)$ is strictly decreasing on $(0,1).$ 
Since for all $x\in (0,1),$ 
\begin{equation}\label{diff2} 
k''(x)=\frac{1}{x^{s+2}} \left(\log\frac{1}{x} \right)^{-2 + t} \left[(-1 + t) t + (t + 2 s t) \log\frac{1}{x} + s (1 + s) \log^2\frac{1}{x}\right]>0,
\end{equation}
$k(x)$ is strictly convex on $(0,1).$ Hence, because the function $k(x)$ satisfies all required properties in Lemma \ref{Optimal}, all minimal $N$-point $K$-energy configurations on $\mathbb{S}^1_{\alpha}$  are configurations of $N$ distinct equally spaced points on $\mathbb{S}^1_{\alpha}$ with respect  to the arc length and vice versa.
\end{proof}

\begin{proof}[Proof of Proposition \ref{Equalspace2}] Let  $N \geq 2,$  $0<\alpha<(e\pi)^{-1},$ and $s,t$ satisfy $s> 0,t\geq 0$ or $s=0,t>0.$ We can use the same lines of reasoning as in the proof of Proposition \ref{Equalspace1} except the function $k$ is considered on $(0,1/e)$ and for all $x\in (0,1/e),$
$$k''(x)=\frac{1}{x^{s+2}} \left(\log\frac{1}{x} \right)^{-2 + t} \left[(-1 + t) t + (t + 2 s t) \log\frac{1}{x} + s (1 + s) \log^2\frac{1}{x}\right]$$
$$\geq \frac{1}{x^{s+2}} \left(\log\frac{1}{x} \right)^{-2 + t} \left[t^2 + 2 s t \log\frac{1}{x} + s (1 + s) \log^2\frac{1}{x}+\left(\log\frac{1}{x}-1\right)t\right]>0.$$ Hence, because the function $k(x)$ satisfies all required properties in Lemma \ref{Optimal}, Proposition \ref{Equalspace2} is proved.
\end{proof}

\begin{proof}[Proof of Proposition \ref{Equalspace3}] Let  $N \geq 2,$ $s\geq 0,$ $t\geq 1,$ and $0<\alpha<1/2.$ Again, we want to use Lemma \ref{Optimal}. The function $k:(0,\pi \alpha]:\rightarrow \mathbb{R}$ in the lemma is 
$$k(x)=\left(\frac{1}{2\alpha \sin(x/2\alpha)}\right)^s \left(\log\frac{1}{2\alpha \sin(x/2\alpha)} \right)^t.$$ Since $2\alpha \sin(x/2\alpha)$ is strictly increasing on $(0,\pi \alpha]$ and $(1/x^s)(\log(1/x))^t$ is strictly decreasing on $(0,1),$ $k(x)$ is strictly decreasing on $(0,\pi \alpha].$ Next, we want to show that $k(x)$ is strictly convex on $(0,\pi \alpha],$ i.e.
\begin{equation}\label{need}
k''(x)>0\quad \textup{for all $x\in (0,\pi \alpha)$}.
\end{equation} To show \eqref{need}, it suffices to show that $q''(x)>0$ for all $x\in (0,\pi/2),$
where 
$$q(x):=\left( \frac{1}{2\alpha \sin x}\right)^s \left( \log \frac{1}{2\alpha \sin x}\right)^{t}.$$
Because for all $x\in (0,\pi/2),$
$$q''(x)=s(\cot^2 x) (2\alpha \sin x)^{-s} \left( \log\left( \frac{1}{2\alpha \sin x}\right)\right)^{t-1}$$
$$+(t-1)(\cot^2 x)(2\alpha \sin x)^{-s} \left(\log\left( \frac{1}{2\alpha \sin x}\right) \right)^{t-2}\left( s\log \left( \frac{1}{2\alpha \sin x}\right)+t\right)$$
$$+ (\csc^2x + s  \cot^2 x) (2\alpha \sin x)^{-s} \left( \log \left( \frac{1}{2\alpha \sin x}\right)\right)^{t-1}\left( s\log\left( \frac{1}{2\alpha \sin x}\right)+t\right)>0,
$$
$k(x)$ is strictly convex on $(0,\pi \alpha].$ Hence, the function $k(x)$ satisfies all required properties in Lemma \ref{Optimal}. This completes the proof. 
\end{proof}

%%%%%%%%%%%%%%%%%%%%%%%%%%%%%%%%%%%%%%%%%%
\section{Discussion and Conclusions}

We introduce minimal $N$-point $s,\log^t$-energy constants
and configurations of an infinite compact metric space $(A,d)$. Such constants and configurations are generated using the kernel
$$K^s_{\log^t}(x,y)=\frac{1}{d(x,y)^s}\left(\log  \frac{1}{d(x,y)}\right)^t,\quad s\geq 0,\quad t\geq 0.$$
In this paper, we study the asymptotic properties of minimal $N$-point $s,\log^t$-energy constants
and configurations of $A$ with $\textup{diam}(A)<1,$ and $N\geq 2$ and $t\geq 0$ are fixed.
We show that the $s,\log^t$-energy $$g(s):=\mathcal{E}^s_{\log^t} (A,N)$$
 is continuous and right differentiable on $[0,\infty)$ and is left differentiable on $(0,\infty)$ in Theorems \ref{contE} and \ref{tptn}. The further analysis on  the differentiability of $g(s)$ can be found in Theorem \ref{diff}. In Theorem \ref{The2}, we show that
\begin{equation*}
\lim_{s \to \infty} \left( \mathcal{E}^{s}_{\log^t} (A,N) \right)^{1/s}=\frac{1}{\delta_N(A)}.
\end{equation*}
and every cluster configuration of $\omega_N^{s,\log^t}$ as $s \to \infty$ is an $N$-point best-packing configuration on $A$. Furthermore, we show in Theorem \ref{clusters0} that for any $s_{0}>0,$ any cluster configuration of $\omega_N^{s,\log^{t}},$ as $s \to s_0,$ is a minimal $N$-point $s_0,\log^{t}$-energy configuration on $A$. When $\textup{diam}(A)<1,$ our theorems generalize Theorems A, B, and C. The natural question would be ``Do Theorems \ref{The2}-\ref{diff} hold true for $\textup{diam}(A)\geq1$?"

Investigation on arrangements of $\omega_N^s$ on circles in $\mathbb{R}^2$ is in Propositions \ref{Equalspace1}-\ref{Equalspace3}.
In these propositions, we show that for certain values of $s$ and $t,$ all minimal $N$-point $s,\log^t$-energy configurations on $\mathbb{S}^1_{\alpha}$ with $\textup{diam}(\mathbb{S}^1_{\alpha})<1$ (corresponding to the Euclidean and geodesic distances) are the configurations of $N$ distinct equally spaced points. We would like to report that the lemma \ref{Optimal} does not allow us to say something when  $\textup{diam}(\mathbb{S}^{1}_{\alpha})\geq 1.$ It would be very interesting to develop a new tool to attack the case  when $\textup{diam}(\mathbb{S}^{1}_{\alpha})\geq 1.$

%%%%%%%%%%%%%%%%%%%%%%%%%%%%%%%%%%%%%%%%%%


%\externalbibliography{yes}
%\bibliography{your_external_BibTeX_file}

%=====================================
% References, variant B: internal bibliography
%=====================================

\end{document}